\documentclass[11pt,reqno,a4paper]{amsart}

\usepackage{amsmath,amsthm,verbatim,amscd,amssymb,setspace,enumitem,hyperref,exscale,color}

\renewcommand{\epsilon}{\varepsilon}
\newcommand{\N}{\mathbb{N}}
\newcommand{\Z}{\mathbb{Z}}
\newcommand{\Q}{\mathbb{Q}}
\newcommand{\R}{\mathbb{R}}
\newcommand{\C}{\mathbb{C}}
\renewcommand{\P}{\mathbb{P}}

\newtheoremstyle{fancy}{}{}{\itshape}{}{\textbf\bgroup}{.\egroup}{ }{}
\newtheoremstyle{fancy2}{}{}{\rm}{}{\textbf\bgroup}{.\egroup}{ }{}

\theoremstyle{fancy}
\newtheorem{theorem}{Theorem}[section]

\newcounter{mtheorem}
\newtheorem{mtheorem}[mtheorem]{Theorem}

\newtheorem{mcor}[mtheorem]{Corollary}

\theoremstyle{fancy2}
\newtheorem{definition}[theorem]{Definition}

\newtheorem{remark}[theorem]{Remark}

\numberwithin{equation}{section}

\begin{document}
\title{Asymptotically conical Calabi-Yau manifolds, III}
\date{\today}
\author{Ronan J.~Conlon}
\address{D\'epartement de Math\'ematiques, Universit\'e du Qu\'ebec \`a Montr\'eal, Case Postale 8888, Succursale
Centre-ville, Montr\'eal (Qu\'ebec), H3C 3P8, Canada}
\email{rconlon@cirget.ca}
\author{Hans-Joachim Hein}
\address{Laboratoire de Math{\'e}matiques Jean Leray, Universit{\'e} de Nantes, 2 rue de la Houssini{\`e}re, 44322 Nantes Cedex 3, France}
\email{hansjoachim.hein@univ-nantes.fr}
\date{\today}
\begin{abstract}
In a recent preprint \cite{ChiLi}, Chi Li proved that asymptotically conical complex manifolds with regular tangent cone at infinity admit holomorphic compactifications (his result easily  extends to the quasiregular case). In this short note, we show that if the open manifold is Calabi-Yau, then Chi Li's compactification is projective algebraic. This has two applications. First, every Calabi-Yau manifold of this kind can be constructed using our refined Tian-Yau type theorem in \cite{Conlon3}. Secondly, we prove classification theorems for such manifolds via deformation to the normal cone.
This includes Kronheimer's classification of ALE spaces and a uniqueness theorem for Stenzel's metric.
\end{abstract}
\maketitle
\markboth{Ronan J.~Conlon and Hans-Joachim Hein}{Asymptotically conical Calabi-Yau manifolds, III}

\section{Introduction}

Let us begin by recalling what we mean by an asymptotically conical (AC) Calabi-Yau manifold.
See \cite[Section 1.3]{Conlon} for details of the following definition.

\begin{definition}\label{ACCY}
Let $(C,g_{0},\Omega_{0})$ be a Calabi-Yau cone with metric $g_{0}$ and holomorphic volume form $\Omega_{0}$. Let $(M,g,\Omega)$ be a Calabi-Yau manifold with metric $g$ and holomorphic volume form $\Omega$. We say that $(M,g,\Omega)$ is \emph{asymptotically conical {\rm (AC)} of rate $\lambda<0$ with asymptotic cone $C$} if there exists a diffeomorphism
$\Phi:C\setminus K\to M\setminus K'$ away from compact sets $K,K'$ such that for all $j \in \N_0$,
\begin{equation*}
|\nabla_{g_0}^j(\Phi^{*}g-g_{0})|_{g_{0}} + |\nabla_{g_0}^j(\Phi^{*}\Omega-\Omega_{0})|_{g_{0}} =O(r^{\lambda-j}).
\end{equation*}
Here, $r$ denotes the radius function of the cone metric $g_{0}$.
\end{definition}

Let us also recall the notion of a {quasiregular} Calabi-Yau cone. Let $D$ be a K\"ahler-Einstein Fano
orbifold, possibly with $\C$-codimension-$1$ singularities. Assume that the total space of the canonical orbibundle $K_D$ is smooth. Moreover, assume that $K_D$ is divisible by $k \in \N$ as a line orbibundle. By the Calabi ansatz, we can endow $C=(\frac{1}{k}K_{D})^{\times}$ (i.e.~the $k$-th root of $K_D$ with its zero section blown down) with the structure of a Calabi-Yau cone $(C, g_0, \Omega_0)$. We call a Calabi-Yau cone constructed in this manner \emph{quasiregular}, and \emph{regular} if $D$ is actually smooth.

Now suppose that $(M,g,\Omega)$ is AC Calabi-Yau with quasiregular asymptotic cone $(C,g_{0},\Omega_{0})$. By \cite[Lemma 2.14]{Conlon}, the complex structures $J$ on $M$ and $J_0$ on $C$ satisfy $|\nabla^{j}_{g_0}(\Phi^{*}J-J_{0})|_{g_{0}}=O(r^{\lambda-j})$ for all $j \in \N_0$.
A very recent result of Li \cite[Theorem 1.2]{ChiLi} (see Appendix \ref{s:chili}) then tells us that $M$ is biholomorphic to $X \setminus D$, where $X$ is a compact complex orbifold without divisorial singularities, containing $D$ as a complex suborbifold with positive normal orbibundle $N_D = -\frac{1}{k}K_D$.

\begin{mtheorem}\label{properties}
The complex orbifold $X = M \cup D$ satisfies the following properties.
\begin{enumerate}
\item[{\rm (i)}] We have that $-K_{X}=(k+1)[D]$ as complex line orbibundles on $X$.
\item[{\rm (ii)}] There exists a holomorphic map $p:X \to Y$ onto a normal projective variety $Y$ such that $p$ is an isomorphism onto its image in a neighbourhood of $D$,  all of the singularities of $Y \setminus p(D)$ are isolated and canonical, the restriction $p|_M: M \to Y \setminus p(D)$ is a crepant resolution of the singularities of $Y \setminus p(D)$, and the $\Q$-Cartier divisor $p_*[D]$ is ample on $Y$.
\item[{\rm (iii)}] $X$ is projective algebraic and satisfies $h^{i,0}(X)=0$ for all $i > 0$.
\item[{\rm (iv)}] Every K\"ahler form on $M$ is cohomologous to the restriction to $M$ of a K\"ahler form on $X$.
\end{enumerate}
\end{mtheorem}

\begin{mcor}
Every {\rm AC} Calabi-Yau manifold with quasiregular asymptotic cone can be obtained from the refined Tian-Yau construction of \cite[Theorem A]{Conlon3}. In particular, if the asymptotic cone
is actually regular, then the optimal value of the rate $\lambda$ is strictly less than $-1$ \cite[Corollary B]{Conlon3}.
\end{mcor}

\begin{remark}
We expect that the optimal value of the rate $\lambda$ is always less than $-1$ whenever the asymptotic cone is quasiregular. Notice also that, at least in the regular case, it essentially follows from \cite[Proposition 1.2 and Theorem 1.2]{ChiLi}, together with our work in \cite{Conlon}, that the optimal value of $\lambda$ is a rational number:~ ${-{\rm min}}\{2, \frac{n}{k}\ell\}$ if the K\"ahler class $\mathfrak{k} \not\in H^2_c(M)$  and
$-{\rm min}\{2n, \frac{n}{k}\ell\}$ if $\mathfrak{k} \in H^2_c(M)$, where $n$ denotes the complex dimension of $M$ and $\ell \in \N \cup \{\infty\}$ is the maximal integer such that $D$ can be $(\ell-1)$-comfortably embedded in a projective manifold $X$ satisfying the assumptions of \cite[Theorem A]{Conlon3}. This statement is rigorous if either $\frac{n}{k}\ell\leq 2n$ and $\mathfrak{k}\in H^2_c(M)$ (which is satisfied for the complete intersections considered in \cite[Section 5]{Conlon}, thereby resolving the conjecture posed there), or $\ell = \infty$ (in which case $M$ is a crepant resolution of $C$; see \cite[Section 4]{Conlon} and \cite[Section 3]{Conlon3}).
\end{remark}

We now explain how to use Theorem \ref{properties} to classify AC Calabi-Yau manifolds. Let $(M, g, \Omega)$ be an AC Calabi-Yau manifold as above. Then the $1$-parameter family $(M, t^2 g, t^{n}\Omega)$ ($t \in \R^+$) converges to $(C, g_0, \Omega_0)$
in the Gromov-Hausdorff sense as $t \to 0$. Very roughly speaking, we now wish to upgrade this metric family to an algebraic one, in order to be able to use that the algebraic deformations of the affine algebraic variety $C$ can in principle be classified. The following theorem achieves this; the Calabi-Yau property \ref{properties}(i) will play a decisive role in the proof (see Remark \ref{r:whatisneeded}).

\begin{mtheorem}\label{thm:deform}
In the setting of Theorem \ref{properties}, there exists a $\C^*$-equivariant deformation $\pi: \mathcal{M} \to \C$ of
the cone $\mathcal{M}_0 = \pi^{-1}(0) = C =  (\frac{1}{k}K_D)^\times$ endowed with its natural $\C^*$-action such that the general fibre of $\pi$ is isomorphic to $Y \setminus p(D)$. Moreover, the deformation $\pi$ has negative grading.
\end{mtheorem}

In particular, $\pi$ is a flat affine morphism of affine algebraic varieties, equivariant with respect to the usual $\C^*$-action on the base and a $\C^*$-action on the total space that restricts to the given
action on the central fibre. See \cite[Chapters 2 and 4]{Pinkham} for the meaning of ``grading'' in this context. 

\begin{mcor}[Kronheimer \cite{Kronheimer2}]\label{cor:Kron}
Every {\rm AC} Calabi-Yau surface is a Kronheimer {\rm ALE} space \cite{Kronheimer}.
\end{mcor}

\begin{proof} Our argument is close in spirit to Kronheimer's, but does not rely on twistor theory.

 The asymptotic cone of an AC Calabi-Yau surface takes the form $C = \C^2/\Gamma$ for some finite group $\Gamma \subset {\rm SU}(2)$ acting freely on $\mathbb{S}^3 \subset \C^2$. Using Theorem \ref{thm:deform}, we may conclude that the affine algebraic surface $Y \setminus p(D)$ is a member of the well-known versal $\C^*$-deformation of the Kleinian singularity $\C^2/\Gamma$ \cite[p.12, Remark 1]{Slodowy},\footnote{Standard results in deformation theory (see \cite[Example 4.5]{ArtinDef} in the algebraic category and \cite{Kas} or \cite[p.198, Theorem]{GrauertDef} in the analytic category) all require shrinking the total space $\mathcal{M}$ of the given deformation in order to be able to pull $\mathcal{M}$ back from the versal deformation. We can undo this shrinking here by using the existence of a $\C^*$-equivariant map to the versal $\C^*$-deformation in the analytic category \cite{Slodowy}. A very similar point appears in \cite[(2.5)]{Kronheimer2}.} so that $M$ is a crepant resolution of one of these deformations.

In turn, \emph{all} deformations of $\C^2/\Gamma$ have negative grading and admit a unique crepant resolution,\footnote{These are well-known special properties of the Kleinian singularities $\C^2/\Gamma$. In general, a Calabi-Yau cone may admit deformations of nonnegative grading and any number of crepant resolutions. See Remark \ref{contrarian} below.}
and every K\"ahler class on such a crepant resolution contains a unique ALE Ricci-flat K\"ahler metric by the results proved and reviewed in \cite{Conlon, Conlon3}, going back to the foundational article \cite{Tian}.
\end{proof}

\begin{mcor}\label{classification}
Let $D$ be a K\"ahler-Einstein Fano manifold. For $k \in \N$ dividing $c_1(D)$, let $M^n$
be an {\rm AC} Calabi-Yau manifold
with asymptotic cone $C = (\frac{1}{k}K_{D})^{\times}$ given by the Calabi ansatz.
\begin{enumerate}
\item[{\rm (i)}] If $D = \P^2$ and $k = 1$, then $M = K_D$ with Calabi's metric \cite{Cal1}.
\item[{\rm (ii)}] If $D = \P^1 \times \P^1$ and $k = 1$, then either $M = K_D$ with one of Goto's deformations of Calabi's metric \cite{goto}, or $M = \P^3\setminus {\rm quadric} =T^*\R\P^3$
with a $\Z_2$-quotient of Stenzel's metric \cite{Stenzel}.
\item[{\rm (iii)}] If $D=\mathbb{P}^{1}\times\mathbb{P}^{1}$ and $k=2$, then $C$ is the $3$-fold ordinary double point, and either $M = T^*\mathbb{S}^3$, the smoothing of $C$, with Stenzel's metric \cite{Stenzel}, or else $M = \mathcal{O}_{\P^1}(-1)^{\oplus 2}$, the small resolution of $C$, with Candelas-de la Ossa's metric \cite{delaossa}.
\item[{\rm (iv)}] If $D$ is a quadric in $\P^n$ with $n > 3$ and $k = n-1$, then $C$ is the $n$-fold ordinary double point, and $M = T^*\mathbb{S}^n$, the smoothing of $C$, with Stenzel's metric \cite{Stenzel}.
\end{enumerate}
\end{mcor}

\begin{proof}
We begin by combining Theorem \ref{thm:deform} with the fact that the cone $C$ in (i) is rigid, whereas the ones in (ii), (iii), (iv) have exactly one deformation, which is smooth. (For (i) this follows from \cite{Schlessinger2}, for (i), (ii), (iii) from \cite{altmann}, and for (iii), (iv) from \cite{Kas}.) Alternatively, we can apply the classification theory of log-Fano varieties \cite[Definition 2.1.1]{AG5} to $Y$:~in (i) and (ii), $(Y,p_*[D])$ is a del Pezzo $3$-fold of degree $9$ and $8$ respectively---by \cite[Remark 3.2.6 and Theorem 3.3.1]{AG5}, the only possible examples are $(\P^3, \mathcal{O}(2))$ and projective cones; in (iii) and (iv), $Y$ must be a quadric by \cite[Theorem 3.1.14]{AG5}, and it is easy to see that a singular quadric with only isolated singularities is a projective cone.

It therefore remains to classify all possible crepant resolutions $M$ of $C$, or at least those carrying a K\"ahler form; in fact, we will use that $M$ is quasiprojective by Theorem \ref{properties}(iii), although this should not be necessary. For (i), (ii), (iii), we begin by observing that $C$ has an obvious crepant resolution $M_0$. By a result of Mori \cite[Theorem 3.5.1]{flops},\footnote{It seems possible that these cases can also be treated using the holomorphic isometries of $C$ instead of Mori theory.} it therefore suffices to classify all possible flops of $M_0$ \cite[Definition 2.2.1]{flops}. In (i) and (ii), $M_0 = K_D$ cannot be flopped because $D$ does not contain any contractible curves, as would be required by \cite[Definition 2.1.1.2]{flops}. In (iii), $M_0 = \mathcal{O}_{\P^1}(-1)^{\oplus 2}$. By \cite[Proposition 2.1.6]{flops}, this has a unique flop, which is isomorphic to $M_0$. Finally, regarding (iv), it is easy to see that $C$ is terminal, so that the blow-down morphism $M \to C$ would have to be small. But \cite[p.2879, footnote]{Conlon} shows that this is not possible because $D$ has Picard rank $1$.

The metric uniqueness statements are now clear by \cite[Theorem 3.1]{Conlon}. As usual, we are ignoring the fact that all of these metrics depend on one or two scaling or diffeomorphism parameters.
\end{proof}

\begin{remark}\label{contrarian}
If $D$ is a K\"ahler-Einstein del Pezzo surface other than $\P^2$ or $\P^1 \times \P^1$, then $k = 1$ and in general, the cone $K_D^\times$ admits many different smooth and singular deformations of negative grading (see for instance \cite[Example 1.4]{Conlon}, but see also \cite[Remark 5.3]{Conlon} for some examples of deformations of \emph{nonnegative} grading) and many different crepant resolutions (see \cite[Example 4.8]{Kollar}; in particular, flops of $K_D$ yield crepant resolutions of $K_D^\times$ whose exceptional set is not pure-dimensional).
\end{remark}

\subsection*{Acknowledgments} We thank Mark Haskins for many helpful discussions over the years, Chi Li for sending us a draft of his paper \cite{ChiLi} and showing us the example in Remark \ref{r:chili}, Richard Thomas for explaining \cite[Example 4.5]{ArtinDef} to us, and Jeff Viaclovsky for pointing out \cite[Lemma 4.1]{LeBrunMaskit}.

\section{Proof of Theorem \ref{properties}}\label{s:properties}

It is clear that $\Omega$ extends to a meromorphic volume form on $X$ with a $(k+1)$-st order pole along $D$ and no poles or zeros elsewhere. Thus, $-K_X = (k+1)[D]$. Given this, Theorem \ref{properties} is an obvious consequence of the following result, which also answers the questions raised in \cite[p.9]{Conlon3}.

\begin{theorem}\label{wh}
Let $X$ be a compact complex orbifold without $\C$-codimension-$1$ singularities. Let $D$ be a suborbifold divisor in $X$ such that $D$ contains the singularities of $X$ and such that the normal orbibundle to $D$ in $X$ is positive. Writing $M = X \setminus D$, the following properties hold.

\begin{enumerate}
\item[{\rm (i)}] There exists a holomorphic map $p:X \to Y$ onto a normal projective variety $Y$ such that $p$ is an isomorphism onto its image in a neighbourhood of $D$, all of the singularities of $Y \setminus p(D)$ are isolated, the restriction $p|_M: M \to Y \setminus p(D)$ is a resolution of singularities for $Y \setminus p(D)$, and the $\Q$-Cartier divisor $p_*[D]$ is ample on $Y$.

\item[{\rm (ii)}] If, in addition, $-K_X = q[D]$ for some $q \in \N$, then $h^{0,i}(X) = 0$ for all $i > 0$. Moreover, all of the singularities of $Y \setminus p(D)$ are canonical.

\item[{\rm (iii)}] If, in addition, $q > 1$ and $M$ admits a K\"ahler form, then $X$ is projective algebraic and every K\"ahler form on $M$ is cohomologous to the restriction to $M$ of a K\"ahler form on $X$.
\end{enumerate}
\end{theorem}

Part (i) was already proved in \cite[Proposition 2.4]{Conlon3}. Parts (ii) and (iii) follow from \cite[Section 2.3]{Conlon3} under the additional assumption that $X$ is K\"ahler. Here, we will in particular establish that $X$ is \emph{necessarily} K\"ahler by going through the relevant proofs in \cite{Conlon3} and showing that they still work if we only assume that $X$ is complex. All of this strongly relies on the fact that $N_D$ is positive.\footnote{Thus, unlike the construction of $X$ as a complex orbifold in \cite{ChiLi} (compare Appendix \ref{s:chili}), our proof that $X$ is K\"ahler is quite different in spirit from the treatment of the asymptotically \emph{cylindrical} case in \cite{HHN}.}

\begin{proof}[Proof of Theorem \ref{wh} in the smooth case]

(i) This was already shown in the proof of \cite[Proposition 2.4]{Conlon3}, based on Grauert's generalisation of the Kodaira embedding theorem \cite[p.343, Satz 2]{Grau:62}.

(ii) It is clear from (i) that $X$ admits $\dim X$ algebraically independent meromorphic functions, so that $X$ is a Moishezon manifold. Using the Hermitian metric on $[D]$ constructed in \cite[Lemma 2.3]{Conlon3} and the vanishing theorems of \cite[Theorem 3 and Appendix]{Riemen2}, we find that
$$H^{0,i}(X) = H^i(X, \mathcal{O}_X) = H^i(X, K_X + q[D]) = 0$$
for all $i > 0$. Alternatively, we can use that $p$ is given by the linear system $|mD|$ for $m \gg 1$; pulling back the Fubini-Study metric then shows that $[D]$ is quasi-positive as defined in \cite[p.266]{GR}, so that the vanishing follows from \cite[Satz 2.1]{GR}. (Both versions of this argument require $X$ to be smooth as we need the sheaf of sections of $[D]$ to be locally free in order to have that $\cdot = \otimes$ in \cite{GR, Riemen2}.)

That the singularities of $Y \setminus p(D)$ are canonical is clear by definition.

(iii) By \cite[Theorem 2.2.18]{Ma1}, the Hodge theorem holds on $X$.
Thus, $H^{2}(X)=H^{1,1}(X)$ by (ii). As in \cite[Proposition 2.5]{Conlon3}, it follows that the restriction map $H^{1,1}(X)\to H^{2}(M)$ is surjective.

Let $\omega$ be any K\"ahler form on $M$. By what we have just said, we can write
$\omega =\xi+d\eta$ for some closed $(1,1)$-form $\xi$ on $X$ and some $1$-form $\eta$ on $M$. Since $d\eta$ is of type $(1,1)$, we may further write $d\eta=i\partial\bar{\partial}u$ for some $u\in C^{\infty}(M)$ by \cite[Corollary A.3(i)]{Conlon}. Now let $\gamma = -i\partial\bar\partial \log h_f$ denote the good curvature form provided by \cite[Lemma 2.3]{Conlon3}, and let $\chi \in C^\infty_0(M)$ be a fixed cut-off function with $d\chi$ supported in the tubular neighbourhood of $D$ where $\gamma > 0$. Then, for $C \gg 1$ sufficiently large,
$$\xi+i\partial\bar{\partial}(\chi u) + C \gamma$$ clearly defines a K\"ahler form on $X$ whose restriction to $M$ is cohomologous to $\omega$.

Now recall that a Moishezon manifold is K\"ahler if and only if it is projective \cite[Theorem 2.2.26]{Ma1}. Alternatively, we can use the fact that $h^{2,0}(X) = 0$ together with Kodaira embedding.
\end{proof}

\begin{proof}[Proof of Theorem \ref{wh} in the general orbifold case] Part (i) is proved for orbifolds in \cite{Conlon3}. If (ii) holds and if the orbifold de Rham cohomology of $X$ satisfies the Hodge decomposition theorem, then the above proof will go through without any changes (using Kodaira embedding in the last step).

Our tool for verifying (ii) and the Hodge decomposition is the existence of a proper modification $\pi: \bar{X} \to X$, given by successive blow-ups with smooth centres not intersecting $D$ at any stage, such that $\bar{X}$ is a projective algebraic orbifold; to find $\pi$, we follow the proof of \cite[Theorem 2.2.16]{Ma1} with $\varphi = p^{-1}$. If $\bar{D} = \pi^{-1}(D)$, then $\bar{X} \setminus \bar{D}$ is smooth, and $\pi$ is an isomorphism onto its image near $\bar{D}$.

Given the map $\pi$, the proof of \cite[Theorem 2.2.18]{Ma1} now goes through in  the orbifold case if one uses orbifold differential forms. This yields the required Hodge decomposition. Moreover, using the injectivity of $\pi^*$ from this proof, (ii) follows if we can show that $H^i(\bar{X}, \mathcal{O}_{\bar{X}}) = 0$ for $i > 0$.

For simplicity, let us write $f = p \circ \pi$. By GAGA, this is an algebraic map. Since $Y$ has canonical singularities away from $f(\bar{D})$ and $f$ is an isomorphism onto its image close to $\bar{D}$, \cite[Theorem 5.22]{KM} yields that $R^qf_*\mathcal{O}_{\bar{X}} = 0$ for $q > 0$. Thus, $H^i(\bar{X}, \mathcal{O}_{\bar{X}}) = H^i(Y, \mathcal{O}_Y)$ by the Leray spectral sequence. It therefore suffices to prove that the latter group vanishes for $i > 0$.\footnote{If the orbifold singularities of $Y$ along $p(D)$ are log-terminal, then this already follows from \cite[Theorem 2.70]{KM}. Our aim here is to prove a version of this theorem that does not require resolving these singularities.}

Assume for the moment that the statement of \cite[Corollary 2.68]{KM}\footnote{Notice that $\omega_Y \otimes M$ should be $\omega_Y \otimes L$ in \cite[Corollary 2.68]{KM}.} holds for $f$ and for the orbifold line bundle $L = M = q[\bar{D}]$, where we have set $a_i = 0$. (This would follow from the statement of \cite[Theorem 2.64]{KM} with $L = M = q[\bar{D}] + f^*H$ for $H$ sufficiently ample on $Y$.) Then $R^q f_*F = 0$ for all $q > 0$, where $F = K_{\bar{X}} + L$. Consequently, from the
Leray spectral sequence, $H^i(Y, f_*F) = H^i(\bar{X}, F)$ for all $i$. Now $F = \sum n_i E_i$ with $n_i \in \N_0$, where the $E_i$ are $f$-exceptional divisors, so that $f_*F = \mathcal{O}_Y$. Moreover, $H^i(\bar{X}, F) = 0$ $(i > 0)$ is exactly the statement of \cite[Theorem 2.64]{KM} with
$L = M = q[\bar{D}]$ there. In conclusion, it remains to show that $H^i(\bar{X}, K_{\bar{X}} + Q) = 0$ for all $i > 0$ for $Q$ an orbifold line bundle with nonnegative curvature on $\bar{X}$ and with strictly positive curvature on a nonempty open subset of $\bar{X}$. But this immediately follows from the proof of \cite[Theorem 6]{Riemen2}.\end{proof}

\section{Proof of Theorem \ref{thm:deform}}\label{s:proof-C}

We first state and discuss a somewhat more technical result: Theorem \ref{thm:deform-tech} below. The following definition, motivated by \cite{Conlon3, Tian} and by Section \ref{s:properties}, will be helpful for this.

\begin{definition}
Let $Y$ be a projective variety. A subset $D \subset Y$ is called an \emph{admissible divisor} if $D$ is a complex hypersurface of $Y$ such that there exists an open complex analytic neighbourhood $U$ of $D$ such that $U$ is a complex orbifold and $D$ is a complex suborbifold of $U$, containing all of the singularities of $U$. We then have an associated $\Q$-Cartier divisor or a $\Q$-line bundle $[D]$ on $Y$.
\end{definition}

Theorem \ref{thm:deform-tech}(i) is essentially a folklore construction based on ``deforming to the normal cone'' \cite{Fulton}; see e.g. \cite[Proposition 5.1]{Conlon} and \cite[p.3]{ChiLi}. We will present this construction carefully here in order to highlight the fact, stated in Theorem \ref{thm:deform-tech}(ii), that the central fibre may not be normal. Another caveat is that we will always ``work within the orbifold category on $U$''; thus, for instance, a \emph{section} of $[D]$ is a section in the usual sense on $Y \setminus D$, where $[D]$ is a genuine line bundle, and is given by invariant sections on the local uniformising charts of the orbifold structure on $U$.

\begin{theorem}\label{thm:deform-tech}
Let $Y$ be a projective variety. Let $D \subset Y$ be an admissible divisor whose associated $\Q$-line bundle $L$ is ample. Let $Y_0$ denote the normal projective variety obtained by contracting the 
$\infty$-section of the $\P^1$-orbibundle $\P(N\oplus \C)$, where $N$ is the normal orbibundle to $D$ in $Y$.
\begin{enumerate}
\item[{\rm (i)}] There exists a test configuration \cite{donaldson, ross-thomas} $\pi: (\mathcal{Y}, \mathcal{L}) \to \C$ with general fibre $(Y,L)$ such that 

$\bullet$ there exists a homeomorphism $F$ from $Y_0$ onto the central fibre $\mathcal{Y}_0$ such that if $v$ denotes the vertex of the projective cone $Y_0$, then $F|_{Y_0 \setminus \{v\}}$ is biholomorphic onto its image, and

$\bullet$ $\mathcal{L}$ is the $\Q$-line bundle associated with a $\C^*$-invariant admissible divisor intersecting the general fibre in $D \subset Y$ and the central fibre in $D = F(\P(0 \oplus \C))$.

\item[{\rm (ii)}] The map $F: Y_0 \to \mathcal{Y}_0$ is the normalisation morphism of $\mathcal{Y}_0$. It is biholomorphic if and only if the restriction map $H^0(Y, L^m) \to H^0(D, N^m)$ is surjective for every $m \in \N$.
\end{enumerate}
\end{theorem}

\begin{remark}\label{r:chili}
The following example shows that the cone $Y_0$ and the central fibre $\mathcal{Y}_0$ need not be isomorphic.
Let $Y$ be a smooth Riemann surface and let $D$ be a point on $Y$. Then clearly $Y_0 = \P^1$. However, if $\mathcal{Y}_0 = \P^1$, then $Y = \P^1$ because the arithmetic genus is constant in flat families.
\end{remark}
 
\begin{proof}[Proof of Theorem \ref{thm:deform}, assuming Theorem \ref{thm:deform-tech}] We consider the variety $Y$ provided by Theorem \ref{properties}(ii). By abuse of notation, we will write $D$ instead of $p(D)$. Then we obtain the deformation $\pi: \mathcal{M} \to \C$ from the test configuration of Theorem \ref{thm:deform-tech} by removing the $\C^*$-invariant admissible divisor of \ref{thm:deform-tech}(i). By \ref{thm:deform-tech}(ii), all that remains to be checked is that the restriction maps $H^0(Y, L^m) \to H^0(D, N^m)$ are surjective for every $m \in \N$, in order to ensure that $\mathcal{M}_0$ is indeed biholomorphic to $C$.

Taking cohomology in the exact sequence $0 \to \mathcal{O}_Y((m-1)D) \to \mathcal{O}_Y(mD) \to \mathcal{O}_{D}(mD) \to 0$, we find that $H^0(\mathcal{O}_Y(mD)) \to H^0(\mathcal{O}_D(mD))$ will be surjective if $H^1(\mathcal{O}_Y((m-1)D)) = 0$. To see that this holds in our case, we simply revisit the second proof of Theorem \ref{wh}; indeed, the same method as in the last paragraph of that proof yields the desired vanishing here as well (substituting $\bar{X} = X$, $\bar{D} = D$, $L = M = -K_{\bar{X}} + (m-1)[\bar{D}]$, and $F = (m-1)[\bar{D}]$).
\end{proof}

\begin{remark}\label{r:whatisneeded}
The two key properties of the pair $(Y,D)$ of Theorem \ref{properties}(ii) that made this proof go through are that $Y \setminus D$ has at worst log-terminal singularities, and that $-K_Y + (m-1)[D]$ is big and nef for all $m \in \N$ (in some suitable orbifold sense). Ultimately these hold because of \ref{properties}(i).
\end{remark}

\begin{proof}[Proof of Theorem \ref{thm:deform-tech}] We begin by reviewing the construction that leads to (i). Let $\hat{\pi}: \hat{\mathcal{Y}} \to \P^1$ be the ``deformation to the normal cone'' \cite[Section 5.1]{Fulton} associated with the pair $(Y,D)$. Thus, $\hat{\mathcal{Y}}$ is the blowup of $\P^1 \times Y$ in $\{0\} \times D$, and $\hat{\pi}$ is the induced projection onto $\P^1$. Then:

$\bullet$ $\hat{\pi}$ is a flat projective morphism.

$\bullet$ All fibres of $\hat{\pi}$ except for the central one, $\hat{\mathcal{Y}}_0 = \hat{\pi}^{-1}(0) =Y \cup \mathbb{P}(N \oplus \C)$, are isomorphic to $Y$.

$\bullet$ The two components of $\hat{\mathcal{Y}}_0$ intersect along $D = \mathbb{P}(N \oplus 0)$.

$\bullet$ $\hat{\pi}$ is equivariant with respect to the natural $\C^*$-action on $\P^1$ and its lift to $\hat{\mathcal{Y}}$.

$\bullet$ The strict transform of $\P^1 \times D$ defines a $\C^*$-invariant admissible divisor $\hat{\mathcal{D}}$.

$\bullet$ $\hat{\mathcal{D}}$ intersects the general fibre in $D \subset Y$ and the central fibre in $D = \P(0 \oplus \C) \subset \P(N \oplus \C)$. 

In order to construct the desired test configuration $\pi: \mathcal{Y} \to \C$, in addition to removing the fibre $\hat{\mathcal{Y}}_\infty = \hat{\pi}^{-1}(\infty)$ from $\hat{\mathcal{Y}}$, we also need to contract the component $Y \subset \hat{\mathcal{Y}}_0$ to a point. This is a local process taking place in a small analytic neighbourhood of this component, and we use here that $L$ is
ample. To see that $Y$ can indeed be contracted, it is helpful to realise that $\hat{\mathcal{Y}}$ can also be written as the blowup of $\mathbb{P}(L \oplus \C)$ in $D \subset Y = \mathbb{P}(0 \oplus \C)$ (with exceptional divisor $\hat{\mathcal{D}} = \mathbb{P}(N \oplus N)$); then $Y \subset \hat{\mathcal{Y}}_0$ equals the preimage of the $\infty$-section $\mathbb{P}(L \oplus 0)$, which can be contracted precisely because $L$ is positive \cite[p.340, Satz 5]{Grau:62}. Part (i) of Theorem \ref{thm:deform-tech} is clear now. 

To prove Part (ii), it suffices to compare the coordinate rings of the two affine algebraic varieties $\mathcal{Y}_0 \setminus D$ and $Y_0 \setminus D$. By construction, $\mathcal{Y}_0 \setminus D$ is the image of $N^* = L^*|_D$ under the contraction map $L^* \to (L^*)^\times = {\rm Spec}\, \bigoplus_{m \in \N_0} H^0(Y, L^m)$, whereas $Y_0 \setminus D = {\rm Spec} \,\bigoplus_{m \in \N_0} H^0(D, N^m)$; see \cite[\S 8.8]{EGA2}. This yields a morphism $Y_0 \setminus D \to \mathcal{Y}_0 \setminus D$, which is an isomorphism if and only if the restriction map $H^0(Y,L^m) \to H^0(D, N^m)$ is surjective for every $m \in \N$. The underlying map of topological spaces is clearly equal to the homeomorphism $F$ of Part (i), which is a biholomorphism off $v$. Since $Y_0 \setminus D$ is normal, \cite[Theorem 6.6]{Rossi2} now tells us that $F$ must be the normalisation map of $\mathcal{Y}_0 \setminus D$. \end{proof}

\appendix

\section{Li's compactification theorem}\label{s:chili}

Let us begin by stating the part of \cite[Theorem 1.2]{ChiLi} that we need.

\begin{theorem}\label{t:chili}
Let $D$ be a compact complex manifold with a holomorphic line bundle $L$. Fix $\delta > 0$ and a Hermitian metric $h$ of positive curvature on $L$. Write $\omega_0 = \frac{i}{2}\partial\bar\partial h^{-\delta}$ for the associated Calabi cone metric with radius function $r$ given by $r^2 = h^{-\delta}$ on $L\setminus 0$. Let $U$ be a tubular neighbourhood of the zero section of $L$, and let $J$ be a complex structure on $U \setminus 0$ such that
\begin{equation}\label{e:chili}|\nabla_{g_0}^j (J - J_0)|_{g_0} = O(r^{\lambda-j})\end{equation}
for some $\lambda < 0$ and all $j \in \N_0$, where $J_0$ denotes the usual complex structure of the total space of $L$. Then, modulo diffeomorphism, $J$ extends to a smooth complex structure on $U$.
\end{theorem}

This result is analogous to \cite[Theorem 3.1]{HHN}, but Li's proof is different from the proof in \cite{HHN}. The aim of this appendix is to clarify the relation between the conical and the cylindrical case.

Fix a small coordinate ball $B \subset D$ and a trivialisation $L|_B \to B \times \C$. Push the set-up forward under this trivialisation, so that $J_0$ now also denotes the product complex structure on $B \times \C$. Let $z$ denote the fibre coordinate and set $\theta = \arg z $. Then
$$g_0 = dr^2 + r^2(d\theta + A)^2 + r^2 g_D,$$
where, up to rescaling, $A$ and $\omega_D$ are the connection and curvature forms of the Chern connection of $(L, h)$ in our chosen trivialisation; more precisely, if we write $h = e^{-\phi}|z|^2$ with $\phi: B \to \R$, then $A = \frac{\delta}{2} d^c\phi$ and $\omega_D = \frac{\delta}{4}dd^c \phi$ with respect to $J_0$. We now define a new Riemannian metric $g_{\sharp}$ by
$$g_\sharp = ds^2 + s^2d\theta^2 + s^2g_D,$$ where $s^2 = |z|^{-2\delta} = e^{-\delta \phi} r^2$.
This is a Riemannian cone metric on $B \times \C$ which is Hermitian with respect to $J_0$, but not K\"ahler. The key point is that $g_0$ and $g_\sharp$ have exactly the same scaling vector field $r\partial_r = s\partial_s$. Thus, by \cite[Lemma 1.6]{Conlon}, $|\nabla_{g_0}^j(g_0 - g_\sharp)|_{g_0} = O(r^{-j})$ and $|\nabla_{g_\sharp}^j(g_0 - g_\sharp)|_{g_\sharp} = O(s^{-j})$ for all $j \in \N_0$. Using this, it is easy to see that
\begin{equation}\label{type2def}
|\nabla_{g_\sharp}^j(J - J_0)|_{g_\sharp} = O(s^{\lambda-j}).
\end{equation}
Next observe that the pullback of $s^{-2}g_\sharp$ under $s = e^t$ is the cylinder metric $g_{\infty} = dt^2 + d\theta^2 + g_D$ on $\R^+ \times \mathbb{S}^1 \times B$, and that the pullback
of $J_0$ is the obvious product complex structure $J_\infty$. Identifying $J$ with its pullback, it then follows from (\ref{type2def}) that $|\nabla^j_{g_{\infty}}(J - J_\infty)|_{g_{\infty}} = O(e^{\lambda t})$. We are now able to construct the desired extension of $J$ by appealing to (an obvious localised version of) \cite[Theorem 3.1]{HHN}, whose proof reduces the problem to the classical results of \cite{NN, NW}.

Up until (\ref{type2def}), this argument is a slight modification of Li's approach; Li instead rewrites (\ref{e:chili}) as a comparison of $J$ and $J_0$ with respect to the smooth metric $\tilde{g}_0 = d\rho^2 + \rho^2 d\theta^2 + g_D$ on $\Delta \times B$, where $\Delta$ denotes the unit disk in $\C$ and $\rho = |z|$ (see \cite[(25)]{ChiLi}). He then constructs $J$-holomorphic coordinates by redoing the analysis of \cite{NN, NW} in weighted H\"older spaces with respect to $\tilde{g}_0$.

Finally, we point out that both arguments obviously extend to the orbifold setting. Alternatively, in order to prove Corollary \ref{cor:Kron}, we could also use a compactification theorem from \cite{LeBrunMaskit}, which relies on twistor theory. In fact, \cite[Lemma 4.1]{LeBrunMaskit} asserts that 
if the cone $(L\setminus 0, \omega_0)$ is flat of dimension $2$ (hence is of the form $\C^2/\Gamma$ for some finite group $\Gamma \subset {\rm U}(2)$ acting freely on $\mathbb{S}^3 \subset \C^2$), and if $J$ is the parallel complex structure of a \emph{scalar-flat} ALE K\"ahler metric of rate $\lambda \leq -\frac{3}{2}$, then $(U \setminus 0, J)$ admits an orbifold compactification obtained by adding on the orbifold curve $D = \mathbb{P}^1/\Gamma$. 

\bibliographystyle{amsplain}
\bibliography{ref}
\end{document}